\documentclass[12pt]{article}
\usepackage{amsthm}
\usepackage{tikz}
\usepackage{fullpage}

\newtheorem{thm}{Theorem}

\newtheorem{cor}[thm]{Corollary}
\newtheorem{con}[thm]{Conjecture}

\begin{document}
\title{New measures of graph irregularity}
\author{Clive Elphick\thanks{\texttt{clive.elphick@gmail.com}}\quad\quad Pawel Wocjan\thanks{Department of Electrical Engineering and Computer Science, University of Central Florida, Orlando, USA; \texttt{wocjan@eecs.ucf.edu}}}
\date{November 3, 2014}
\maketitle

Dedicated to the late Dr.~C.~S.~Edwards who supervised the Ph.D. of C.~Elphick

\abstract{In this paper, we define and compare four new measures of graph irregularity. We use these measures to prove upper bounds for the chromatic number and the Colin de Verdi\`ere parameter. We also strengthen the concise Tur\'an theorem for irregular graphs and investigate to what extent Tur\'an's theorem can be similarly strengthened for generalized $r$-partite graphs. We conclude by relating these new measures to the Randi\'c index and using the measures to devise new normalised indices of network heterogeneity.} 

\section{Introduction}

Many results in extremal graph theory are exact only for some regular graphs. In this paper we strengthen various bounds, using two degree based measures of irregularity and two spectral measures of irregularity, so that they also become exact for some irregular graphs.     

Let $G$ be a simple and undirected graph with vertex set $V$ with $|V| = n$, edge set $E$ with $|E| = m$, $t$ triangles, clique number $\omega$, chromatic number $\chi$ and vertex degrees $\Delta = d_1 \ge d_2 \ge ...  \ge d_n = \delta.$ Let $\mu$ denote the largest eigenvalue of the adjacency matrix of $G$, let $q$ denote the largest eigenvalue of the signless Laplacian and let $d$ denote the average degree.

Existing measures of irregularity include the following. Collatz and Sinogowitz \cite{collatz57} proposed a spectral measure, namely $\mu - d$.  Bell \cite{bell92} proposed a variance measure, namely $\mathrm{var}(G) = \sum(d_i - d)^2/n = \sum(d_i^2/n) - d^2$ and identified the most irregular graphs for both measures. He also showed that the measures are incomparable for some pairs of graphs. Albertson \cite{albertson97} used the measure $\sum_{ij\in E} |d_i - d_j|$, which has found applications in chemical graph theory. Nikiforov \cite{nikiforov06} used the measure $s(G) = \sum_i |d_i - d|$. These measures are all greater than or equal to zero, with equality for regular graphs, and can be described as $\emph{additive}$ measures of irregularity. It is worth noting that $\mathrm{var}(G) = \mathrm{var}(\overline{G})$ and $s(G) = s(\overline{G})$, where $\overline{G}$ denotes the complement of G. The measures defined in this paper are all greater than or equal to one, with equality for regular graphs, and can be described as $\emph{multiplicative}$ measures of irregularity.

In Section 2 we define and compare the new measures; in Section 3 we use the measures to prove an upper bound for the chromatic number; in Section 4 we strengthen Tur\'an's Theorem for irregular graphs; in Section 5 we apply the new measures to generalised $r-$partite graphs; in Section 6 we bound a graph's radius, Harmonic index and Randi\'c index; and we conclude with bounds for the new measures and new indices of network heterogeneity.   

\section{Measures of irregularity}

Our first measure of irregularity, $\nu$, was introduced by Edwards \cite{edwards77}. He defined a parameter $c_v$, which he termed the ``vertex degree coefficient of variation''  as follows:  

\[
\nu = 1 + c_v^2 = \frac {n\sum_{i\in V} d_i^2}{4m^2}.
\]

Edwards \cite{edwards77} showed that $c_v = 0$ if and only if a graph is regular, so $\nu \ge 1$, with equality only for regular graphs. $c_v$ is the ratio of the standard deviation to the mean of the vertex degrees, which follows the usual definition of a coefficient of variation.
 
Our second measure of irregularity, $\epsilon$, is defined similarly using an "edge degree coefficient of variation" as follows:

\[
 \epsilon = 1 + c_e^2 =\frac{n\sum_{ij\in E} \sqrt{d_i d_j}}{2m^2}.
\]

It follows from Proposition 2.8 in Favaron, Mah\'eo and Sacl\'e \cite{favaron93} that $\epsilon \ge 1$, with equality only for regular graphs.  

Finally we define two spectral measures of irregularity as follows:

\[
\beta = \frac{\mu}{d} = \frac{\mu n}{2m} \mbox{  and  } \gamma = \frac{q}{2d} = \frac{qn}{4m}.
\]

It is well known that $\mu \ge d$, with equality only for regular graphs. Therefore $\beta \ge 1$, with equality only for regular graphs. Similarly it follows from Proposition 3.1 in Cvetkovic, Rowlinson and Simic \cite{cvetkovic07} that $\gamma \ge 1$ with equality only for regular graphs.

We can compare these bounds as follows. Hofmeister \cite{hofmeister88} proved that $\mu^2 \ge \sum_{i \in V} d_i^2 /n$ and Favaron \emph{et al} \cite{favaron93} have proved that $\mu \ge \sum_{ij \in E} \sqrt{d_i d_j}/m$. It is therefore straightforward that:

\[
 \beta^2 \ge \nu \mbox{   and   }  \beta \ge \epsilon.
\]

We can also show that $\nu \ge \epsilon$, as follows:

\[
\nu =  \frac {n\sum_{i\in V} d_i^2}{4m^2} = \frac{n\sum_{ij \in E}(d_i + d_j)}{4m^2} \ge \frac{n\sum_{ij\in E} \sqrt{d_i d_j}}{2m^2} = \epsilon.
\]

For most but not all irregular graphs, $\epsilon^2 > \nu$. 

It is well known that $q \ge 2\mu$ so $\gamma \ge 2\beta$, with equality only for regular graphs. We can also show that $\gamma \ge \nu$ as follows. Liu and Liu \cite{liu09} proved the following inequality for connected graphs but in fact their proof does not use the connectedness assumption:

\[
\sum_{i\in V} d_i^2 \le mq.
\]

Therefore:

\[
\gamma = \frac{nq}{4m} \ge \frac{n\sum d_i^2}{4m^2} = \nu.
\]

$\gamma$ and $\beta^2$ are incomparable, but based on a sample of named graphs in Wolfram Mathematica, the average value of these two measures of irregularity are about equal. Wheels provide examples of graphs for which $\gamma >> \beta^2$. The relationship between $\beta$ and $\epsilon$ is therefore comparable to that between $\gamma$ and $\nu$.

\section{Upper bounds for the Chromatic Number}

\begin{thm}

Let G be a graph with irregularity $\beta$. Then 

\[
\chi(G) \le \frac{n}{\beta}.
\]

\end{thm}

\begin{proof}
It is well known that $\chi(\chi - 1) \le 2m$ and Edwards and Elphick \cite{edwards83} have proved that $\mu^2 \le 2m(\chi - 1)/\chi$. Therefore:

\[
\mu^2 \le \frac{2m(\chi - 1)}{\chi} \le \frac{4m^2}{\chi^2}.
\]

Hence:

\[
\chi \le \frac{2m}{\mu} = \frac{n}{\beta}.
\]

\end{proof}

This bound strengthens a bound due to Hansen and Vukicevi\'c \cite{hansen09} , who recently proved that $\chi(G) \le 2R(G)$, where $R(G)$ is the Randi\'c index. We discuss this index and the Harmonic index, $H(G)$,  in Section 6.    An alternative strengthening of the bound due to Hansen and Vukicevi\'c \cite {hansen09} has been provided by Deng et al \cite{deng13}, who proved $\chi(G) \le 2H(G).$

\subsection{Colin de Verdi\`ere parameter}

The Colin de Verdi\`ere parameter, $\lambda(G)$, is the basis for the profound conjecture that $\chi(G) \le 1 + \lambda(G)$. There is extensive literature on this conjecture, for example by Holst et al \cite{holst99} and Goldberg \cite{goldberg10}. Several upper bounds for $\chi(G)$ are not upper bounds for $1 + \lambda(G)$.   For example, the Petersen graph demonstrates that $\lambda \not\le \mu$ and $K_{4,5}$ demonstrates that $\lambda \not\le n - \alpha$, where $\alpha$ denotes the independence number, which is the size of the maximum set of vertices, no two of which are adjacent.

We can, however, use $\beta$ to create a new upper bound for $\lambda$ as follows.

\begin{thm}
Let G be a connected graph with irregularity $\beta$. Then:

\[
\lambda(G) \le \frac{n}{\beta} - 1 = \frac{2m}{\mu} - 1.
\]

\end{thm}

\begin{proof}

One of the deep properties of $\lambda$ is that it is minor-monotone, from which it follows immediately that $\omega \le 1 + \lambda$. (A graph parameter $\phi(G)$ is called minor-monotone if $\phi(H) \le \phi(G)$ for any minor $H$ of $G$.)    

Pendavingh \cite{pendavingh98}  has proved that if $G \not= K_{3,3}$ is a connected graph, then

\[
 \lambda(\lambda + 1) \le 2m.
\]

We therefore need to consider two options. If $G = K_{3,3}$ then (eg see Goldberg) $\lambda = 4 < (2m/\mu) - 1 = 18/3 - 1 = 5$.

If $G \not= K_{3,3}$ then we use a result due to Nikiforov \cite{nikiforov02}, and conjectured by Edwards and Elphick \cite{edwards83}, that:

\[
\mu^2 \le \frac{2m(\omega - 1)}{\omega}.
\]

Therefore:

\[
\mu^2 \le \frac{2m(\omega - 1)}{\omega} \le \frac{2m\lambda}{\lambda + 1} \le \frac{4m^2}
{(\lambda + 1)^2},  
\]

and consequently:

\[
\lambda \le \frac{2m}{\mu} - 1 = \frac{n}{\beta} - 1.
\]

\end{proof}

\section{Tur\'an's Theorem for irregular graphs}

Tur\'an's Theorem, proved in 1941, is a fundamental result in extremal graph theory. In its concise form it states that:

\[
2m \le \frac{(\omega - 1)n^2}{\omega}.
\]

Observe that the result, due to Nikiforov above, that $\mu^2 \le 2m (\omega -1)/\omega$, is equivalent to the following strengthening of the concise Tur\'an theorem:  

\begin{thm}
\[
2m \le \frac{(\omega - 1)n^2}{\omega\beta^2}.
\]

Due to the bounds $\beta^2 \ge \nu$ and $\beta \ge \epsilon$ we obtain:
 
\[
 \mbox{(i)   } 2m \le \frac{(\omega - 1)n^2}{\omega\nu} \mbox{   and  (ii)   } 2m \le \frac{(\omega - 1)n^2}{\omega\epsilon^2}.  
\]
\end{thm}

We provide a non-spectral proof of the bound (i) because it  leads to a corollary. Before presenting this proof we explain briefly the intuition underlying the above inequalities.  Theorem 3 is unusual because it involves $m$ on both sides. A useful way to interpret the theorem is that $\beta$,  $\nu$  and $\epsilon$ are measures of graph irregularity. Therefore all graphs with a given clique number and, for example, irregularity as measured by $\nu \ge 2$ have a maximum number of edges that is at most half of the number implied by Tur\'an's Theorem. 

\begin{proof}

This non-spectral proof is based on a 1962 proof of the concise Tur\'an Theorem due to Moon and Moser \cite{moon62}, as written up in an award winning paper by Martin Aigner entitled ``Tur\'an's Graph Theorem''.   

Let $C_h$ denote the set of $h$-cliques in $G$ with $|C_h| = c_h$. So for example, $c_1 = n, c_2 = m, c_3 = t$ etc. For $A \in C_h$ let $d(A)$ equal the number of $(h+1)$ cliques containing $A$. Moon and Moser \cite{moon62} proved that:

\begin{equation}\label{eq:moon1}
\frac{c_{h+1}}{c_h} \ge \frac{h^2c_h/c_{h-1} - n}{h^2 - 1}, \quad\mbox{for }  h\ge 2.
\end{equation}  

They also proved that:
\[
nc_h + (h^2 - 1) c_{h+1} \ge \sum_{B \in C_{h-1}} d(B)^2
\]
so with $h = 2$ this becomes:
\begin{eqnarray}
nm + 3c_3         & \ge & \sum_{i=1}^n d_i^2, \quad \mbox{or equivalently}  \nonumber \\
\frac{c_3}{c_2} & =   & \frac{c_3}{m} \ge \frac{(\sum d_i^2/m) - n}{3}. \label{eq:moon2}
\end{eqnarray}

Now define $\theta$ as follows:
\begin{equation}\label{eq:moon3}
\frac{(\theta - 2)n}{\theta} = \frac{\sum d_i^2}{m} - n
\end{equation}
which is equivalent to:
\[
2m = \frac{(\theta - 1)n^2}{\theta\nu}.
\]

This definition of $\theta$ differs from that in \cite{moon62} and enables the strengthening of Moon and Moser's proof. Combining (\ref{eq:moon2}) and (\ref{eq:moon3}) we have:

\begin{equation}\label{eq:moon4}
\frac{c_3}{c_2} \ge \frac{\sum d_i^2/m - n}{3} = \frac{(\theta - 2)n}{3\theta}.
\end{equation}

To prove Theorem 3 (i) we need to show that $\theta \le k - 1$ for graphs without $k$-cliques. Consider the claim:

\begin{equation}\label{eq:moon5}
\frac{c_{h+1}}{c_h} \ge \frac{(\theta - h)n}{\theta (h+1)}, \quad \mbox{for }  h\ge 2.
\end{equation}

For $h = 2$, this is inequality (\ref{eq:moon4}). We therefore use induction on $h$ and (\ref{eq:moon1}) as follows:

\begin{eqnarray*}
\frac{c_{h+1} }{c_h } & \ge & \frac{h^2c_h/c_{h-1} - n }{h^2 - 1 } 
\ge \frac{h^2(\theta - h +1)n/(\theta h) - n }{h^2 - 1 } \\
&=& \frac{(\theta -h)(h - 1)n}{\theta(h^2 - 1)} = \frac{(\theta - h)n}{\theta(h + 1)}
\end{eqnarray*}
as claimed in (\ref{eq:moon5}). Now if $G$ contains no $k$-clique then $c_k = 0$ and we infer $\theta \le h = k - 1$ from (\ref{eq:moon5}).  

We have not attempted a non-spectral proof of the $\epsilon$ bound. 
\end{proof}

\subsection{Bounds using $\gamma$}

In Theorems 1, 2 and 3 it is straightforward to prove that $\beta$ can be replaced with $\sqrt{\gamma}$. In each case use the following bound due to He, Jin and Zhang \cite{he13} and Abreu and Nikiforov \cite{abreu13} :

\[
\frac{2n}{2n - q} \le \omega \le \chi.
\]

\subsection{Number of $k$-cliques}

Moon and Moser \cite{moon62} proved that if t is the number of triangles in a graph, then:

\[
t \ge \frac{m(4m - n^2)}{3n}.
\]

In the following corollary we strengthen this bound for irregular graphs. This corollary is exact for some irregular complete tripartite and Tur\'an graphs.

\begin{cor}

Let G be a graph with irregularity $\nu$. Then:

\[
t \ge \frac{m(4m\nu - n^2)}{3n}.
\]

\end{cor}

\begin{proof}

From inequality (\ref{eq:moon4}), we know that:

\[
t \ge \frac{nm(\theta - 2)}{3\theta} =  \frac{\sum d_i^2 - nm}{3} = \frac{4\nu m^2 - n^2m}{3n}.
\]

\end{proof}

This approach can be continued for larger cliques. For example, we know from (\ref{eq:moon5}) that:

\[
c_4 \ge \frac{tn(\theta -3)n}{4\theta} \ge \frac{m(4m\nu - n^2)}{12}(1 - \frac{3}{\theta}) = \frac{m(4m \nu - n^2)(3m \nu - n^2)}{6n^2}.
\]

\subsection{Remarks}
Theorem 3 is exact for all complete bipartite graphs. The full form of Tur\'an's theorem states that $m(G) \le m(T_r(n))$, where $T_r(n)$ is the complete $r$-partite graph of order $n$ whose classes differ by at most one, with equality holding only if $G = T_r(n)$. It is not the case that for all irregular graphs $m(G) \le m(T_r(n))/\nu$ or that $m(G) \le m(T_r(n))/\epsilon^2$ or that $m(G) \le m(T_r(n))/\beta^2$.

 Tur\'an's theorem can be further strengthened by using more complex lower bounds for $\mu$. For example, if $t_i$ denotes the sum of the degrees of the vertices adjacent to $v_i$, then Yu, Lu and Tian \cite{yu04} have proved that:

\[
\mu^2 \ge \frac{\sum t_i^2}{\sum d_i^2} \ge \frac{\sum d_i^2}{n} \ge \frac{4m^2}{n^2}.
\]

Thus if we define:

\[
\alpha = \frac{n^2 \sum t_i^2}{4m^2 \sum d_i^2} \ge \nu
\]

 it follows, as above, that:

\[
2m \le \frac{(\omega - 1)n^2}{\omega \alpha} \le \frac{(\omega - 1)n^2}{\omega \nu}.
\]

\section{Generalized $r$-partite graphs}

In a series of papers, Bojilov and others have generalized the concept of an $r$-partite graph. They define the parameter $\phi$ to be the smallest integer $r$ for which $V(G)$  has an $r$-partition:

\[
V(G) = V_1 \cup V_2 \cup \ldots \cup V_r, \quad\mbox{such that } d(v) \le n - |V_i|,
\] 
for all $v \in V_i$ and for $i = 1,2, \ldots ,r.$

It is notable that $\phi$ depends only on the degrees of G, and not on the adjacency matrix of G. Indeed, $\phi$ is defined for any set of $n$ integers $a_i$, where $0 \le a_i \le n - 1$, which may or may not correspond to the degrees of a graph. 

Theorem 2.1 in \cite{bojilovcaro12} proves that $\phi$ is a lower bound for the clique number and the greedy Algorithm 1 and Theorem 3.1 in \cite{bojilovcaro12} demonstrate that $\phi$ can be computed in linear time. For $d$-regular graphs, Theorem 4.4 in \cite{bojilovcaro12} proves that:

\[
\phi = \left\lceil\frac{n}{n - d}\right\rceil.
\]

Khadzhiivanov and Nenov \cite{khadzhiivanov04} have proved that $\phi$ satisfies Tur\'an's Theorem:

\begin{equation}\label{bojilov1}
2m \le \frac{(\phi - 1)n^2}{\phi} \le \frac{(\omega - 1)n^2}{\omega}.
\end{equation}

Theorem 4.1 in \cite{bojilovcaro12} provides a simpler proof of (\ref{bojilov1}). The study of $\phi$ has therefore led to a novel proof of the concise version of Tur\'an's Theorem, which also demonstrates that this famous result is in fact a function only of the degrees of a graph rather than its adjacency matrix.

It is of interest to see to what extent (\ref{bojilov1}) can be strengthened in a similar way to Theorem 3. For example, Bojilov and Nenov \cite{bojilovnenov12} have strengthened (\ref{bojilov1}) as follows:

\begin{equation}\label{bojilov2}
2m \le \frac{(\phi - 1)n^2}{\phi\sqrt\nu}.
\end{equation}

Inequality (\ref{bojilov2}) is further strengthened in Theorem 5.4 in \cite{bojilovcaro12} where it is shown that:

\[
\phi \ge \frac{n}{n - d^*_\phi} \ge \frac{n}{n - d^*_{\phi - 1}} \ge \ldots \ge \frac{n}{n - d^*_1}
\]
where
\[
d^*_r = \sqrt[r]{\sum d_i^r/n}.
\]

Observe that inequality (\ref{bojilov2}) is equivalent to $r = 2$ in this chain of inequalities.

 It is therefore natural to ask whether $2m \le (\phi - 1)n^2/\phi\nu$? The answer is no, because, for example, the graph in Figure~\ref{fig:graph751}
 provides a counter-example. 

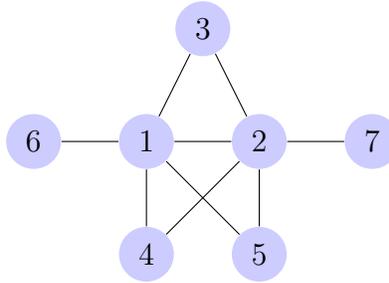
\begin{figure}[h]
\begin{center}
\begin{tikzpicture}
  [scale=.3,every node/.style={circle,fill=blue!20}]
  \node (n1) at (7.5,10) {3};
  \node (n2) at (0,5)  {6};
  \node (n3) at (5,5)  {1};
  \node (n4) at (10,5) {2};
  \node (n5) at (15,5)  {7};
  \node (n6) at (5,0)  {4};
  \node (n7) at (10,0)  {5};

  \foreach \from/\to in {n1/n3,n1/n4,n2/n3,n3/n4,n4/n5,n3/n6,n3/n7,n4/n6,n4/n7}
    \draw (\from) -- (\to);
\end{tikzpicture}
\end{center}
\caption{{\small \texttt{Graph751} on $7$ vertices with degree sequence $(5, 5, 2, 2, 2, 1, 1)$, $\phi = 2$ and $\omega = 3$}}
\label{fig:graph751}
\end{figure}

There are also various spectral lower bounds for $\omega$ of which the simplest, due to Cvetkovic \cite{cvetkovic72}, is:

\begin{equation}\label{bojilov3}
\omega \ge \frac{n}{n - \mu}.
\end{equation}

The graph in Figure~\ref{fig:graph2} is  an example of a graph which does not satisfy (\ref{bojilov3}), with $\omega$ replaced by $\phi$. It also demonstrates that a variety of other spectral lower bounds for $\omega$ are not lower bounds for $\phi$.  Furthermore, $\phi$ does not satisfy the Motzkin-Straus inequality. 

\begin{figure}[h]
\begin{center}
\begin{tikzpicture}
  [scale=.3,every node/.style={circle,fill=blue!20}]
  \node (n1) at (7.5,10) {1};
  \node (n2) at (0,5)  {4};
  \node (n3) at (5,5)  {2};
  \node (n4) at (10,5) {3};
  \node (n5) at (15,5)  {5};
  \node (n6) at (5,0)  {6};
  \node (n7) at (10,0)  {7};

  \foreach \from/\to in {n1/n2,n1/n3,n1/n4,n1/n5,n2/n3,n3/n4,n4/n5,n2/n6,n3/n6,n4/n7,n5/n7,n6/n7}
    \draw (\from) -- (\to);
\end{tikzpicture}
\end{center}
\caption{{\small \texttt{Graph} on $7$ vertices with degree sequence $(4, 4, 4, 3, 3, 3, 3)$, $\phi = 2$, $\mu =3.503$}}
\label{fig:graph2}
\end{figure}
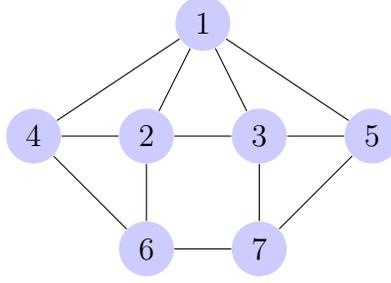

\begin{con}

We have, however,  been unable to find a counter-example to the following conjecture:

\[
2m \le \frac{(\phi - 1)n^2}{\phi\epsilon}.
\]

\end{con}

\section{Bounds on graph radius, Harmonic and Randi\'c indices}

The Randi\'c index is used in organic chemistry, with bonds between atoms represented by edges in a graph. The Randi\'c index is defined as:

\[
R(G) = \sum_{ij\in E} \frac{1}{\sqrt{d_i d_j}}.
\]
  
An alternative to the Randi\'c index is the Harmonic index, which is defined as:

\[
H(G) = \sum_{ij\in E} \frac{2}{d_i + d_j}.
\]

Using results due to Xu \cite{xu12} it is straightforward to show that:

\[
\frac{n}{2\nu} \le H(G) \le R(G) \le \frac{n}{2}.
\]

Liu \cite{liu13} has recently proved that triangle-free graphs have $H(G) \ge 2m/n$. We can generalise this bound using part (i) of Theorem 3 as follows:

\[
H(G) \ge \frac{n}{2\nu} \ge \frac{\omega m}{(\omega - 1) n} = \frac{2m}{n} \mbox { when  } \omega = 2.
\]

We can also show that:

\[
\frac{m}{\mu} = \frac{n}{2\beta} \le \frac{n}{2\epsilon} \le R(G)
\]

since using Cauchy-Schwartz, we have $R(G).\sum_{ij\in E} \sqrt{d_i d_j} \ge (\sum_{ij\in E} 1)^2 = m^2$. Note that for Star graphs, $n/2\epsilon = \sqrt{n -1}$, which is the lower bound for $R(G)$ due to Bollobas and Erdos \cite{bollobas98}. It is not always the case that $n/2\epsilon \le H(G)$ or that $n/2\beta \le H(G)$.

The eccentricity $ecc(v)$ of a vertex $v$ in a connected graph $G$ is the maximum distance between $v$ and any other vertex $u$ of $G$. The minimum graph eccentricity is the radius, $r$, of the graph.    Xu \cite{xu12} has proved that if $H(G)$ is the Harmonic index, then:

\[
H(G) \ge \frac{m}{n - r}.
\]

This bound on $r$ can be strengthened as follows.

\begin{thm}

Let $G$ be a connected graph with irregularity $\nu$. Then:

\[
 H(G) \ge \frac{n}{2\nu} \ge \frac{m}{n - r}.
\]

\end{thm}

\begin{proof}

Note that for each vertex $i \in V(G)$, we have $d_i \le n - ecc(i)$. Therefore:

\[
\frac{n}{2\nu} = \frac{2m^2}{\sum_{ij \in E}(d_i + d_j)} \ge \frac{2m^2}{\sum_{ij \in E}(2n - ecc(i) - ecc(j))} \ge \frac{2m^2}{2m(n - r)} = \frac{m}{n - r}.
\]

\end{proof}

\section{Upper Bounds}

Gutman, Furtula and Elphick \cite{gutman14} have proved that:

\[
\beta^2 \le \frac{n^2}{4(n - 1)} \mbox{  ;  } \nu \le \frac{n^2}{4(n - 1)} \mbox{  and  } \epsilon^2 \le \frac{n^2}{4(n - 1)}
\]

with equality for Star graphs. 

We can prove the same upper bound for $\gamma$ as follows.

\begin{thm}
Let $G$ be a connected graph with $n$ vertices. Then

\[
\gamma \le \frac{n^2}{4(n - 1)}.
\]
This bound is exact for Star graphs which have $q = n$.
\end{thm}

\begin{proof}

Proposition 15(5) in Hansen and Lucas \cite{hansen10} states that for connected graphs:

\[
q \le \frac{dn^2}{2(n - 1)} = \frac{mn}{n - 1}.
\]

Therefore

\[
\gamma = \frac{nq}{4m} \le \frac{mn^2}{4m(n - 1)} = \frac{n^2}{4(n - 1)}.
\]

\end{proof}

We can obtain alternative bounds on $\nu(G)$ by using bounds on $\sum d_i^2$, which is often referred to as the first Zagreb index. For example, Das \cite{das04} proved that

\[
 \sum d_i^2 \le 2m(\Delta + \delta) - n\Delta\delta.
\]
 
Therefore

\[
1 \le \nu = \frac{n\sum d_i^2}{4m^2} \le \frac{n(\Delta + \delta)}{2m} - \frac{n^2\Delta \delta}{4m^2} = \frac{\Delta + \delta}{d} - \frac{\Delta \delta}{d^2}.
\]

Alternatively, Izumino, Mori and Seo \cite{izumino98} have proved (their Corollary 3.2) that if $0 \le \delta \le d_i \le \Delta$ then:

\[
\frac{1}{n}\sum d_i^2 - \left(\frac{\sum d_i}{n}\right)^2 \le \frac{(\Delta - \delta)^2}{4}.
\]

Therefore

\[
1 \le \nu = \frac{n\sum d_i^2}{4m^2} \le \frac{n^2}{4m^2}\left(\frac{4m^2}{n^2} + \frac{(\Delta - \delta)^2}{4}\right) = 1 + \left(\frac{\Delta - \delta}{2d}\right)^2.
\]

We can obtain alternative bounds on $\epsilon(G)$ by using bounds on the generalised Randic index, $R_\alpha(G)$, with $\alpha = 1/2$. For example, Li and Yang \cite{li04} proved that for $\alpha \ge 0$:

\[
R_\alpha \le \frac{n(n-1)^{1 + 2\alpha}}{2}.
\]

Therefore

\[
1 \le \epsilon = \frac{nR_{1/2}}{2m^2} \le \frac{n^2(n - 1)^2}{4m^2} = \left(\frac{n - 1}{d}\right)^2.
\] 
 
Favaron \emph{et al} \cite{favaron93}  demonstrate that $\nu$ and $\epsilon^2$ are incomparable. However, in practice, $\epsilon^2 \ge \nu$ for almost all graphs.  Indeed we have been unable to find a graph amongst the named graphs in Wolfram Mathematica for which $\epsilon^2 < \nu$.  Considering the irregular named graphs in Wolfram with 16 vertices, the average value of $\nu = 1.22$, the average value of $\epsilon^2 = 1.27$ and the average value of $\beta^2 = 1.32$. As a specific example, DutchWindmill(5,4) has $\nu =1.6$, $\epsilon^2 = 1.675$ and $\beta^2 = 1.92$.

The graphs representing some actual networks, such as the World Wide Web, power grids, academic collaborators and neural networks, are highly irregular. For example, Newman \cite{newman02} calculated that  the World Wide Web graph has $c_v = 3.685$, implying $\nu = 14.6$.  These high values of irregularity for some actual networks may increase the usefulness of the measures described in this paper.     

 \section{Network heterogeneity indices}

Estrada \cite{estrada} and others have noted that many real-world networks have a power law degree distribution. Estrada has proposed that normalised indices of the heterogeneity of such networks should lie in the range $(0, 1)$, with zero corresponding to regular graphs and unity to Star graphs. Estrada devised the following index, using $R(G)$,  which meets these criteria:

\[
\rho_n = \frac{n - 2R}{n - 2\sqrt{n - 1}}.
\]

As discussed above, $\nu, \epsilon$ and $\beta$ are minimised for regular graphs and maximised for Star graphs.    We can therefore devise the following normalised heterogeneity indices:

\[
\nu_n = \frac{n^2 - (n^2/\nu)}{(n - 2)^2} \mbox{  ;  } \epsilon_n = \frac{n - (n/\epsilon)}{n - 2\sqrt{n - 1}} \mbox{  and  } \beta_n = \frac{n - (n/\beta)}{n - 2\sqrt{n - 1}}.
\]

It follows from the inequalities in Section 6 above that:

\[
0 \le \rho_n \le \epsilon_n \le \beta_n \le 1.
\]

It may be that $\beta_n$ is the most useful of these indices, perhaps using results due to Chung, Lu and Vu \cite{chung03} who have investigated the spectrum of random power law graphs. 

\section*{Acknowledgements}
P.W. gratefully acknowledges the support of the National Science Foundation CAREER Award CCF-0746600. This work was supported in part by the National Science Foundation Science and Technology Centre for Science of Information, under grant CCF-0939370.


\begin{thebibliography}{10}

\bibitem{abreu13}
N. M. M. de Abreu and V. Nikiforov, \emph{Maxima of the $Q$-Index: Graphs with bounded clique number}, El. J. Lin. Algebra, 26, (2013), 121 - 130.

\bibitem{albertson97}
M. O. Albertson, \emph{The irregularity of a graph}, Ars Comb., 46 (1997), 219 -- 225.
\bibitem{bell92}
F.~Bell, \emph{Note on the irregularity of graphs}, Linear Algebra and Appl., 161 (1992), 45 -- 54. 

\bibitem{bojilovcaro12}
A.~Bojilov, Y.~Caro, A.~Hansberg and N.~Nenov, \emph{Partitions of graphs into small and large sets}, arXiv:1205:1727, (2012).

\bibitem{bojilovnenov12}
A.~Bojilov and N.~Nenov, \emph{An inequality for generalized chromatic graphs, Mathematics and education in mathematics}, (2012), 143 -- 147.

\bibitem{bollobas98}
B. Bollob\'as and P. Erdos, \emph{Graphs of extremal weights}, Ars Combin., 50 (1998), 225 -- 233. 

\bibitem{chung03}
F. Chung, L. Lu and V. Vu, \emph{Eigenvalues of Random Power Law Graphs}, Annals of Combin., 7 (2003), 21 -- 33. 

\bibitem{collatz57}
L.~Collatz and U.~Sinogowitz, \emph{Spektren endlicher Grafen.} Abh. Math. Sem. Univ. Hamburg 21 (1957), 63 -- 77.

\bibitem{cvetkovic72}
D.~Cvetkovic, \emph{Chromatic number and the spectrum of a graph}, Publ. Inst. Math. (Beograd), 14(28), (1972), 25 -- 38.

\bibitem{cvetkovic07}
D. Cvetkovic, P. Rowlinson and S. Simic, \emph{Signless Laplacians of finite graphs}, Lin. Algebra and Appl., 423, (2007), 155 - 171.

\bibitem{das04}
K. C. Das, \emph{Maximising the sum of the squares of the degrees of a graph}, Discrete Math., 285 (2004), 57 -- 66.

\bibitem{deng13}
H. Deng, S. Balachandran, S. K. Ayyaswamy and Y. B. Venkatakrishnan, \emph{On the harmonic index and the chromatic number of a graph}, Discrete Applied Mathematics, (2013), in press.

\bibitem{estrada}
E. Estrada, \emph{The Structure of Complex Networks : Theory and Applications}, Oxford University Press, (2011).

\bibitem{khadzhiivanov04}
N.~Khadzhiivanon and N.~Nenov, \emph{Generalized $r$-partite graphs and Tur\'an's Theorem}, Compt. Rend. Acad. bulg. Sci, 57 (2004).

\bibitem{edwards77}
C.~S.~Edwards, \emph{The largest vertex degree sum for a triangle in a graph}, Bull. London Math. Soc., 9,
 (1977), 203--208.

\bibitem{edwards83}
C.~S.~Edwards and C.~Elphick, \emph{Lower bounds for the clique and the chromatic numbers of a graph}, Discrete Appl. Math, 5, (1983), 51--64.  

\bibitem{favaron93}
O. Favaron, M. Mah\'eo and J. -F. Sacl\'e, \emph{Some eigenvalue propoerties in Graphs (conjectures of Graffiti - II}, Discrete Mathematics, 111 (1993), 197 -- 220. 

\bibitem{goldberg10}
F. Goldberg, \emph{The Colin de Verdi\`ere number of a graph}, Ph.D. thesis, (2010).

\bibitem{gutman14}
I. Gutman, B. Furtula and C. Elphick, \emph{Three New/Old Vertex-Degree-Based Topological Indices}, MATCH Commun. Math. Comput. Chem., 72, 3, (2014), 617--632.

\bibitem{hansen09}
P. Hansen and D. Vukicevi\'c, \emph{On the Randi\'c index and the chromatic number}, Discrete Mathematics, 309, (2009), 4228 -- 4234.

\bibitem{hansen10}
P. Hansen and C. Lucas, \emph{Bounds and conjectures for the signless Laplacian index of graphs}, Lin. Algebra and Appl., 432, (2010), 3319 - 3336.

\bibitem{he13}
B. He, Y-L Jin and X. Zhang, \emph{Sharp bounds for the signless Laplacian spectral radius in terms of clique number}, Lin. Algebra and Appl., 438, (10), (2013), 3851 - 3861.

\bibitem{hofmeister88}
M.~Hofmeister, \emph{Spectral radius and degree sequence}, Math. Nachr., 139 (1988) 37--44.

\bibitem{holst99}
H. van der Holst, L. Lovasz and A. Schrijver, \emph{The Colin de Verdiere graph parameter}, vol. 7 of Bolyai Soc. Math. Stud. (1999), 29 -- 85.

\bibitem{izumino98}
S. Izumino, H. Mori and Y. Seo, \emph{On Ozeki's inequality}, J. of Inequal. and Appl., 2, (1998), 235 -- 253.

\bibitem{li04}
X. Li and Y. Yang, \emph{Sharp Bounds on the General Randi\'c Index}, MATCH Commun. Math. Comput. Chem., 51, (2004), 155 -- 166.

\bibitem{liu13}
J. Liu, \emph{On the Harmonic index of Triangle-Free Graphs}, Applied Mathematics, 4, (2013), 1204 -- 1206.

\bibitem{liu09}
M. Liu and B. Liu, \emph{New sharp upper bounds for the first Zagreb index}, MATCH Commun. Math. Comput. Chem., 62(3), (2009), 689 - 698.

\bibitem{moon62}
J.~W.~Moon and L.~Moser, \emph{On a problem of Tur\'an}, Publ. Math. Inst. Hungar. Acad. Sci. 7, (1962), 283--286.

\bibitem{newman02}
M. E. J. Newman, \emph{Random graphs as models of networks}, Sante Fe Institute, Working Paper 2002-02-005.

\bibitem{nikiforov02}
V.~Nikiforov, \emph{Some inequalities for the largest eigenvalue of a graph}, Combin. Probab, Comput. 11 (2002), 179--189.

\bibitem{nikiforov06}
V.~Nikiforov, \emph{Eigenvalues and degree deviation in graphs}, Linear Algebra Appl., 414, (2006), 347--360.

\bibitem{nikiforov07}
V. Nikiforov, \emph{Bounds on graph eigenvalues I}, Linear Algebra Appl., 420, (2007), 667 -- 671.

\bibitem{li10}
X. Li and Y. Shi, \emph{On a relation between the Randi\'c index and the chromatic number}, Discrete Mathematics, Volume 310, 17, (2010), 2448 -- 2451.  

\bibitem{pendavingh98}
R. Pendavingh, \emph{On the relation between two minor-monotone graph parameters}, Combinatorica, 18, 2 (1998), 281 -- 292.

\bibitem{xu12}
X. Xu, \emph{Relationships between Harmonic index and other Topological indices}, Applied Mathematical Sciences, 6, 41, (2012), 2013 -- 2018.

\bibitem{yu04}
A. Yu, M. Lu and F. Tian, \emph{On the spectral radius of graphs}, Linear Algebra Appl., 387, (2004), 41 -- 49.

\end{thebibliography}
\end{document}